\documentclass[11pt,reqno]{amsart}
\usepackage{tikz}
\usepackage[center]{caption}
\usepackage{xcolor}

\usepackage{amssymb,latexsym,amsmath}
\usepackage{graphicx,mathrsfs}
\usepackage{hyperref,url}
\hypersetup{colorlinks=true, urlcolor=black, linkcolor=black, citecolor=black}
\usepackage{amstext}
\usepackage{bbm}
\usepackage{amsthm}
\usepackage{dsfont}
\usepackage{amsfonts}

\newtheorem{thm}{Theorem}[section]

\newtheorem{lemma}{Lemma}[section]

\newtheorem{definition}{Definition}[section]
\newtheorem{corollary}{Corollary}[section]

\newtheorem{claim}{Claim}[section]

\numberwithin{equation}{section}

\newcommand{\R}{\mathbb{R}}

\newcommand{\Z}{\mathbb{Z}}

\newcommand{\Sone}{S^1}
\newcommand{\C}{\mathbb{C}}
\newcommand{\HH}{L^2(\mathbb{R})}

\newcommand{\intR}{\int_{\mathbb{R}}}
\newcommand{\eqae}{\overset{\textit{a.e.}}{=}}
\newcommand{\Rstar}{\R\backslash\{0\}}
\newcommand{\LL}{\mathscr{L}(\HH)}
\newcommand{\UU}{\mathscr{U}(\HH)}
\newcommand{\HHR}{L^2_R(\mathbb{R})}
\newcommand{\Rplus}{(0,\infty)}
\newcommand{\allf}{\quad f\in\HH}
\newcommand{\alls}{\quad s\in\R}
\newcommand{\aexi}{\quad a.e.\ \xi\in\R}
\newcommand{\ab}{\alpha,\beta}

\begin{document}

\title[Semistable unitary operators on $L^2(\mathbb{R})$]
{\small{A note on semistable unitary operators on $L^2(\mathbb{R})$}}
\author{Xianghong Chen}
\address{Xianghong Chen, Department of Mathematics, Sun Yat-sen University, Guangzhou, 510275, P.R. China}
\email{chenxiangh@mail.sysu.edu.cn}
\date{\today}
\dedicatory{} 
\commby{} 


\begin{abstract}
In this note, we present a characterization of semistable unitary operators on $L^2(\mathbb{R})$, under the assumption that the operator is 
(i) translation-invariant, 
(ii) symmetric, and 
(iii) locally uniformly continuous (LUC) under dilation. 

As a consequence, we characterize one-parameter groups formed by such operators, which are of the form 
$e^{i\beta t|{d}/{dx}|^\alpha}$, with $\alpha,\beta\in\mathbb R$. 
\end{abstract}

\maketitle

\tableofcontents


\section{Introduction}
\label{sec:intro}
\vspace{1em}

\subsection{Notation}
\label{subsec:notation}

In this note, by ``measurable'' we always mean ``Lebesgue measurable''. We use $\Sone$ to denote $\{z\in\C: |z|=1\}$. 
The Hilbert space $\HH$ is defined as
\begin{align*}
\big\{f:\R\rightarrow\C \ \big|\ \text{$f$ is measurable, } \|f\|_2^2=\intR |f(x)|^2 dx<\infty\big\}
/\{f: \|f\|_2=0\}.
\end{align*}
For $f,g\in\HH$, we write $f=g$ if $[f]=[g]$, that is, $f\eqae g$; $\langle f,g\rangle$ denotes 
$$\intR f(x)\overline{g(x)}dx.$$ 
Denote by $\LL$ the space of bounded linear operators from $\HH$ to itself, and denote 
$$\UU=\{T\in\LL: T \text{ is unitary}\}.$$
For $f\in\HH$, $\lambda\in\R\backslash\{0\}$, and $a\in\R$, denote 
$$\delta_\lambda f(x)=\frac{1}{|\lambda|}f\big(\frac{x}{\lambda}\big), $$
$$\tau_a f(x)=f(x-a).$$ 
For $T\in\LL$ and $\lambda\in\Rstar$, denote 
$$T_\lambda=\delta_\lambda\, T\,\delta_{\lambda^{-1}}.$$

For $f\in\HH$, the Fourier transform is defined by 
$$\hat f(\xi)=\frac{1}{\sqrt{2\pi}}\intR f(x)e^{-i\xi x},\quad \xi\in\R.$$
For $R>1$, define 
$$\HHR=\{f\in\HH: \hat f(\xi)=0,\ a.e.\ \xi \text{ with } |\xi|<R^{-1} \text{ or }|\xi|>R\}.$$

\subsection{Definitions}
\label{subsec:definitions}

We now collect some necessary definitions for the main results. The first definition is motivated by the notion of semistable distribution in probability theory. 

\begin{definition}[semistable operator]
\label{def:semistable}
An operator $T\in\LL$ is called \emph{semistable} if there exist constants $a, b>0$, such that 
\begin{align*}
\begin{cases}
T^2=T_a,\\
T^3=T_b.
\end{cases}
\end{align*}
\end{definition}

The next definition will be a key technical condition, in which ``local'' should be understood on the Fourier side, and ``uniform'' should be understood in terms of operator topology. 

\begin{definition}[local uniform continuity, LUC]\label{def:LUC}
\ \\
$(i)$ An operator family $\{T(\lambda)\in\LL\}_{\lambda\in\Rplus}$ is called \emph{locally uniformly continuous} if for any $R>1$, the restricted operator 
$$T(\lambda): \HHR\rightarrow\HH$$
is continuous in $\lambda\in\Rplus$ with respect to the operator norm. \\
$(ii)$ An operator $T\in\LL$ is called \emph{locally uniformly continuous (LUC) under dilation} if the operator family $\{T_\lambda\}_{\lambda\in\Rplus}$ is locally uniformly continuous. 
\end{definition}

The following definitions are standard. 

\begin{definition}
\label{def:translation-inv}
Let $T$ be an operator in $\LL$. \\
$(i)$ $T$ is called \emph{translation-invariant} if 
$\tau_{-a}\, T\, \tau_{a}=T$ holds for all $a\in\R$. \\
$(ii)$ $T$ is called \emph{symmetric} if 
$T_{-1}=T$. 
\end{definition}

Finally, we recall the definition of weakly measurable operator group. 

\begin{definition}
\label{def:op-group}
$(i)$ An operator family $\{T(t)\in\LL\}_{t\in\R}$ is called an \emph{operator group} if it satisfies $T(0)=I\, (\text{identity operator})$ and 
\begin{equation}\label{eq:group-property-0}
T(t_1+t_2)=T(t_1)\,T(t_2),\quad t_1, t_2\in\R.
\end{equation}
$(ii)$ An operator family $\{T(t)\in\LL\}_{t\in\R}$ is called \emph{weakly measurable} if for any $f, g\in\HH$, the function 
$t\mapsto\langle T(t)f,g\rangle$ 
is measurable on $\R$. 
\end{definition}

\subsection{Statements of results}
\label{subsec:results}

The first main result of this note gives a characterization of semistable unitary operators on $L^2(\mathbb{R})$. 

\begin{thm}
\label{thm:1}
Let $T\in\UU$ be a semistable unitary operator. If $T$ is 
$(i)$ translation-invariant, 
$(ii)$ symmetric, and 
$(iii)$ LUC under dilation, 
then there exist constants $\alpha, \beta\in\R$, such that 
\begin{equation}\label{eq:T-hat}
\widehat{Tf}(\xi)=e^{i\beta|\xi|^\alpha}\hat f(\xi),\quad f\in\HH.
\end{equation}
Moreover, if $T\neq I$, then the constants $\alpha, \beta$ are uniquely determined by $T$.
\end{thm}

It is easy to see that if $T\in\LL$ is defined by \eqref{eq:T-hat} with $\alpha\neq0$, 
then $T$ satisfies the conditions of Theorem \ref{thm:1}. 
Thus, in view of the uniqueness of $\ab$, the following definition is validated. 

\begin{definition}[Schr\"odinger operator of order $\alpha$] 
\label{def:schrodinger-op}
Suppose $T\in\UU$ satisfies \eqref{eq:T-hat} with $\alpha\in\Rstar$ and $\beta\in\R$. Define 
$$\alpha(T)=\begin{cases}
    0,\,\quad\text{if } T=I, \\
    \alpha,\quad\text{if } T\neq I.   
\end{cases}$$
We call $T$ a \emph{Schr\"odinger operator of order} $\alpha(T)$. 
\end{definition}

The following corollary is an easy consequence of Theorem \ref{thm:1}. 

\begin{corollary}
\label{cor:1}
Suppose $T\in\UU$ satisfies the conditions of Theorem \ref{thm:1} and $T\neq I$. Then $T$ is a Schr\"odinger operator of order $\alpha$ if and only if 
$$T^2=T_{2^{1/\alpha}}.$$ 
\end{corollary}

The second main result of this note gives a characterization of operator groups consisting of semistable unitary operators on $L^2(\mathbb{R})$.  

\begin{thm}
\label{thm:2}
Let $\{T(t)\in\UU\}_{t\in\R}$ be a weakly measurable operator group. Suppose that for each $t\in\R$, $T(t)$ satisfies the conditions of Theorem \ref{thm:1}. Then there exist constants $\alpha, \beta\in\R$, such that 
\begin{equation}\label{eq:T(t)-hat}
\widehat{T(t)f}(\xi)=e^{i\beta t|\xi|^\alpha}\hat f(\xi),\quad f\in\HH,\ t\in\R.
\end{equation}
Moreover, if $T(t)\not\equiv I$, then the constants $\ab$ are uniquely determined by $\{T(t)\}_{t\in\R}$. 
\end{thm}

Note that if $\{T(t)\}_{t\in\R}$ is defined by \eqref{eq:T(t)-hat} with $\alpha\neq0$, 
then $\{T(t)\}_{t\in\R}$ satisfies the conditions of Theorem \ref{thm:2}. 

Note also that if $\{T(t)\}_{t\in\R}$ satisfies the conditions of Theorem \ref{thm:2} and $T(t)\not\equiv I$, then by Theorem \ref{thm:1}, $T(t)\neq I$ for all $t\neq 0$; moreover, the order of $T(t)$ is constant for $t\neq0$. This leads to the following definition. 

\begin{definition}[Schr\"odinger group of order $\alpha$]
\label{def:schrodinger-group}
Suppose $\{T(t)\}_{t\in\R}$ satisfies the conditions of Theorem \ref{thm:2}. 
Define the order of $\{T(t)\}_{t\in\R}$ as the order $($say $\alpha$$)$ of $T(1)$, and call $\{T(t)\}_{t\in\R}$ a \emph{Schr\"odinger group of order} $\alpha$. 
\end{definition}

It follows from Theorem \ref{thm:2} that if $\{T(t)\}_{t\in\R}$ is a Schr\"odinger group of order $\alpha\neq0$, then $\{T(t)\}_{t\in\R}$ satisfies 
$$T(t)=T(1)_{t^{1/\alpha}},\quad t>0.$$
The following corollary is direct consequence of Theorems \ref{thm:1} and \ref{thm:2}. 

\begin{corollary}
\label{cor:2}
Suppose $\{T(t)\}_{t\in\R}$ is a Schr\"odinger group and $T(t)\not\equiv I$. Then the following are equivalent: \\
$(i)$ $\{T(t)\}_{t\in\R}$ is of order $\alpha$; \\
$(ii)$ $T(t)=T(1)_{t^{1/\alpha}}$, for all $t>0$; \\
$(iii)$ $T(t)=T(1)_{t^{1/\alpha}}$, for some $t>0$, $t\neq 1$.
\end{corollary}


\section{Preliminaries and lemmas}
\label{sec:prelim}

In this section, we collect some facts and lemmas that will be used in the proofs of Theorem \ref{thm:1} and Theorem \ref{thm:2}. 

We start with the following standard fact (cf. \cite[Theorem~1.5]{Hormander}). 

\begin{lemma}
\label{lem:translation-inv}
A unitary operator $T\in\UU$ is translation-invariant if and only if 
there exists a measurable function $m:\R\rightarrow\Sone$, such that 
$$\widehat{Tf}(\xi)=m(\xi)\hat f(\xi),\allf.$$
\end{lemma}

It is easy to see that the function $m(\xi)$ is uniquely determined by $T$ in a.e. sense. 

The next lemma is a simple consequence of Lemma \ref{lem:translation-inv} and the identity
$$\widehat{f_\lambda}(\xi)=\hat f(\lambda\xi).$$ 

\begin{lemma}
\label{lem:semistable-multiplier}
Let $T$ and $m(\cdot)$ be as in Lemma \ref{lem:translation-inv}. Then\\
$(i)$ $T$ is symmetric if and only if 
$$m(\xi)=m(|\xi|),\quad a.e.\ \xi\in\R.$$
$(ii)$ $T$ is semistable if and only if there exist constants $a, b>0$, such that 
$$\begin{cases}
m(a\xi)=m(\xi)^2,\\
m(b\xi)=m(\xi)^3,\aexi. 
\end{cases}$$
\end{lemma}

The LUC condition in Theorem \ref{thm:1} is related to the following. 

\begin{lemma}
\label{lem:LUC-m}
Let $T$ and $m(\cdot)$ be as in Lemma \ref{lem:translation-inv}. Suppose that $T$ is symmetric. Then $T$ is LUC under dilation if and only if there exists a continuous function 
$\widetilde m: \Rplus\rightarrow\Sone$, such that 
$$m(\xi)=\widetilde m(|\xi|),\aexi.$$
\end{lemma}

It is easy to see that the function $\widetilde m$ is uniquely determined by $T$. 
The proof of Lemma \ref{lem:LUC-m} is based on the following. 

\begin{lemma}
\label{lem:LUC-phi}
Suppose $\varphi: \R\rightarrow\C$ is bounded, measurable, and satisfies 
\begin{equation}\label{eq:unif-cont-varphi}
    \lim_{\varepsilon\rightarrow0} \|\varphi(\cdot+\varepsilon)-\varphi(\cdot)\|_{L^\infty[-M,M]}=0, 
\end{equation}
for all $M>1$. Then there exists a unique continuous function $\widetilde \varphi:\R\rightarrow\C$, such that 
$$\varphi(s)=\widetilde\varphi(s),\quad a.e.\ s\in\R.$$
\end{lemma}

\begin{proof}
Denote by $E$ the set of Lebesgue points of $\varphi$. 
Since $E$ is dense in $\R$, it suffices to show that $\varphi$ is uniformly continuous on $E\cap [-\frac M2,\frac M2]$ for all $M>1$. 
Indeed, consider 
$$\varphi_\delta(s):=\frac{1}{2\delta}\int_{s-\delta}^{s+\delta}\varphi(t)dt,\quad \delta>0,\ s\in\R.$$
Then for any $\delta<\frac M2$ and $s_1, s_2\in[-\frac M2,\frac M2]$, we have 
$$|\varphi_\delta(s_1)-\varphi_\delta(s_2)|\le \|\varphi(\cdot+|s_1-s_2|)-\varphi(\cdot)\|_{L^\infty[-M,M]}.$$
Taking $\delta\rightarrow 0$, we see that for any $s_1, s_2\in E\cap [-\frac M2,\frac M2]$, 
$$|\varphi(s_1)-\varphi(s_2)|\le \|\varphi(\cdot+|s_1-s_2|)-\varphi(\cdot)\|_{L^\infty[-M,M]}.$$
By \eqref{eq:unif-cont-varphi}, it follows that $\varphi$ is uniformly continuous on $E\cap [-\frac M2,\frac M2]$. 

The uniqueness of $\widetilde\varphi$ is easy. 
\end{proof}

Now we prove Lemma \ref{lem:LUC-m}. 

\begin{proof}[Proof of Lemma \ref{lem:LUC-m}]
The sufficiency part is easy. Since 
$$\widehat{T_\lambda f}(\xi)-\widehat{T_{\lambda_0} f}(\xi)=[m(\lambda\xi)-m(\lambda_0\xi)]\,\hat f(\xi),\allf,$$
by standard argument (cf. \cite{Hormander}), we have 
\begin{equation}\label{eq:Deltam}
\|T_\lambda -T_{\lambda_0}\|_{\HHR\rightarrow\HH}=\|m(\lambda\cdot)-m(\lambda_0\cdot)\|_{L^\infty(R^{-1}\le|\xi|\le R)},\quad R>1.
\end{equation}
If $m(\xi)\eqae\widetilde m(|\xi|)$ for a continuous function $\widetilde m: \Rplus\rightarrow\Sone$, then by the local uniform continuity of $\widetilde m$, we have for any $\lambda_0>0$ and $R>1$, 
$$\|m(\lambda\cdot)-m(\lambda_0\cdot)\|_{L^\infty(R^{-1}\le|\xi|\le R)}=\|\widetilde m(\lambda\cdot)-\widetilde m(\lambda_0\cdot)\|_{L^\infty[R^{-1},R]}\rightarrow0,$$
as $\lambda\rightarrow\lambda_0$. 
By \eqref{eq:Deltam} and Definition \ref{def:LUC}(ii), this shows that $T$ is LUC under dilation. 

Now suppose $T$ is LUC under dilation. 
Then by Definition \ref{def:LUC}(ii) and \eqref{eq:Deltam} (take $\lambda_0=1$), we have for any $R>1$, 
$$\lim_{\lambda\rightarrow1}\|m(\lambda\cdot)-m(\cdot)\|_{L^\infty[R^{-1},R]}=0.$$
Write
$$\varphi(s)=m(e^s),\alls.$$
Then $\varphi$ satisfies the conditions of Lemma \ref{lem:LUC-phi}. 
So there exists a continuous function $\widetilde \varphi:\R\rightarrow\Sone$, such that 
$$\varphi(s)=\widetilde \varphi(s),\quad a.e.\ s\in\R.$$
Set 
$$\widetilde m(r)=\widetilde\varphi(\ln r),\quad r>0.$$
Then $\widetilde m: \Rplus\rightarrow\Sone$ is continuous, and satisfies 
$$m(\xi)=\widetilde m(|\xi|),\aexi.$$
This proves the necessity. 
\end{proof}

The next lemma is a standard topology exercise. It will be used to ``lift'' continuous functions from $\R$ to $\Sone$, to continuous functions from $\R$ to $\R$. 

\begin{lemma}
\label{lem:circle-unwrap}
Suppose $\varphi: \R\rightarrow\Sone$ is a continuous function with $\varphi(0)=e^{i\beta_0}$. Then there exists a unique continuous function $\phi:\R\rightarrow\R$, such that $\phi(0)=\beta_0$ and 
$$\varphi(t)=e^{i\phi(t)},\quad t\in\R.$$
\end{lemma}

The following lemma is a well-known fact. It will be used in the proofs of Lemma \ref{lem:dense-steps} and Theorem \ref{thm:2} below. 

\begin{lemma}
\label{lem:steinhaus}
Suppose $\phi: \R\rightarrow\R$ is a measurable function that satisfies 
$$\phi(t_1+t_2)=\phi(t_1)+\phi(t_2),\quad t_1, t_2\in\R. $$
Then 
$$\phi(t)=\phi(1)\,t,\quad t\in\R.$$
\end{lemma}

The following lemma will be used in the proof of Lemma \ref{lem:HKV}, with $\phi$ being a continuous function. 

\begin{lemma}
\label{lem:dense-steps}
Suppose $\phi: \R\rightarrow\R$ is a measurable and locally integrable function. If there exists a dense subset $Q\subset\R$, such that for any $q\in Q$, 
\begin{equation}\label{eq:q-step}
\phi(s+q)-\phi(s)=\mathrm{const},\quad a.e.\ s\in\R,
\end{equation}
then there exist unique constants $\alpha, \gamma\in\R$, such that 
$$\phi(s)=\alpha s+\gamma,\quad a.e.\ s\in\R.$$
\end{lemma}

\begin{proof}
First consider the case where $\phi$ is continuous. By \eqref{eq:q-step} and the continuity of $\phi$, we have 
\begin{equation}\label{eq:c(q)}
\phi(s+q)-\phi(s)=c(q),\alls,\ q\in Q, 
\end{equation}
for a function $c: Q\rightarrow\R$. 
Applying \eqref{eq:c(q)} with $s=0$, we have 
\begin{equation}\label{eq:c(q)-2}c(q)=\phi(q)-\phi(0),\quad q\in Q.
\end{equation}
Plugging \eqref{eq:c(q)-2} into \eqref{eq:c(q)}, we obtain 
\begin{equation}\label{eq:c(q)-3}
\phi(s+q)-\phi(s)=\phi(q)-\phi(0),\alls,\ q\in Q.
\end{equation} 
By the continuity of $\phi$, it follows that \eqref{eq:c(q)-3} holds for all $q\in\R$. 
Write $\widetilde\phi(s)=\phi(s)-\phi(0)$. In terms of $\widetilde\phi$, \eqref{eq:c(q)-3} reads 
$$\widetilde\phi(s+q)=\widetilde\phi(s)+\widetilde\phi(q),\quad s, q\in\R.$$
By Lemma \ref{lem:steinhaus}, it follows that $\widetilde\phi(s)=\widetilde\phi(1)s$. 
Consequently, 
$$\phi(s)=(\phi(1)-\phi(0))s+\phi(0),\alls.$$

Now assume that $\phi$ is measurable and locally integrable. Then by \eqref{eq:q-step}, 
\begin{equation}\label{eq:c(q)-4}
\phi(s+q)-\phi(s)=c(q),\quad a.e.\ s\in\R, 
\end{equation}
for a function $c: Q\rightarrow\R$. 
Integrating over $(s-\delta,s+\delta)$, \eqref{eq:c(q)-4} implies 
\begin{equation}\label{eq:c(q)-5}
\phi_\delta(s+q)-\phi_\delta(s)=c(q),\alls,\ \delta>0, 
\end{equation}
where 
$$\phi_\delta(s)=\frac{1}{2\delta}\int_{s-\delta}^{s+\delta}\phi(t)dt.$$ 
Since $\phi_\delta:\R\rightarrow\R$ is continuous, by the continuous case proven above, there exist constants $\alpha_\delta, \gamma_\delta\in\R$, such that 
\begin{equation}\label{eq:alpha(delta)}
\phi_\delta(s)=\alpha_\delta s+\gamma_\delta,\alls.
\end{equation}
Let $s_1\neq s_2$ be two Lebesgue points of $\phi$. 
Since 
$$\lim_{\delta\rightarrow0}\phi_\delta(s_j)=\phi(s_j),\quad j=1,2,$$
applying \eqref{eq:alpha(delta)} with $s=s_j$ ($j=1,2$), we see that 
$$\lim_{\delta\rightarrow0}\alpha_\delta=\alpha,\quad \lim_{\delta\rightarrow0}\gamma_\delta=\gamma, $$
for some constants $\alpha, \gamma\in\R$. 
By \eqref{eq:alpha(delta)}, it follows that 
$$\lim_{\delta\rightarrow0}\phi_\delta(s)=\alpha s+\gamma,\alls.$$
On the other hand, we have 
$$\lim_{\delta\rightarrow0}\phi_\delta(s)=\phi(s)$$
at every Lebesgue point $s$ of $\phi$. Therefore, 
$$\phi(s)=\alpha s+\gamma,\quad a.e.\ s\in\R.$$
The uniqueness of $\alpha, \gamma$ follows by a similar argument. 
\end{proof}

Finally, we record a simple lemma, which will be further extended in Lemma \ref{lem:triple}. 

\begin{lemma}
\label{lem:single-double}
Suppose $\alpha_j\in\R$, $\beta_j\in\Rstar$, $j=1, 2$.\\
$(i)$ If $e^{i\beta_1 r^{\alpha_1}}\equiv 1,\ r>0$, then 
$$\alpha_1=0,\;\;\beta_1\in 2\pi\Z.$$
$(ii)$ If $e^{i\beta_1 r^{\alpha_1}}e^{i\beta_2 r^{\alpha_2}}\equiv 1,\ r>0$, then either 
\begin{align*}
    &(a)\ \alpha_1=\alpha_2=0,\;\; \beta_1+\beta_2\in 2\pi\Z,\;\;\text{or}\\
    &(b)\ \alpha_1=\alpha_2\neq 0,\;\; \beta_1+\beta_2=0.
\end{align*}
\end{lemma}

\begin{proof}
(i) If $\alpha_1=0$, then $e^{i\beta_1}=1$, which clearly implies 
$\beta_1\in 2\pi\Z.$ 
If $\alpha_1\neq 0$, then the set 
$$\{r>0: {\beta_1 r^{\alpha_1}}\in2\pi\Z\}$$
is countable, thus $e^{i\beta_1 r^{\alpha_1}}\equiv 1$ does not hold. This proves (i). 

(ii) If $\alpha_1=\alpha_2=0$, then by (i), $\beta_1+\beta_2\in2\pi\Z$. If $\alpha_1=\alpha_2\neq 0$, then again by (i), we have $\beta_1+\beta_2=0$. 
If $\alpha_1\neq \alpha_2$, then the set 
$$\{r>0: {\beta_1 r^{\alpha_1}}+{\beta_2 r^{\alpha_2}}\in2\pi\Z\}$$
is countable, thus 
$e^{i\beta_1 r^{\alpha_1}}e^{i\beta_2 r^{\alpha_2}}\equiv 1$ does not hold. 
This proves (ii).
\end{proof}


\section{Proof of Theorem \ref{thm:1}}
\label{sec:thm1}

In this section, we prove Theorem \ref{thm:1}. 
The proof relies on the following lemma from Hamedani-Key-Volkmer \cite[Theorem~2.3]{HKV}, where the authors considered more general complex-valued functions and semistability conditions. 
Here we limit ourselves to the case of uni-modular functions 
and make a small observation that the irrationality condition on $\frac{\ln b}{\ln a}$ can be relaxed in our setting. 

\begin{lemma}
\label{lem:HKV}
Suppose $m:\Rplus\rightarrow\Sone$ is a continuous function. If there exist constants $a,b>0$ such that 
\begin{align}
    \begin{cases}
    m(ar)=m(r)^2,\\
    m(br)=m(r)^3,\quad r>0.\label{eq:m(r)}
    \end{cases}
\end{align}
Then there exist constants $\ab\in\R$, such that 
\begin{equation}\label{eq:m(r)-formula}
m(r)=e^{i\beta r^\alpha},\quad r>0.
\end{equation}
Moreover, if $m(r)\not\equiv 1$, then the constants $\ab, a, b$ are uniquely determined by $m(\cdot)$ and satisfy $a^\alpha=2,\;\;b^\alpha=3.$
\end{lemma}

\begin{proof}
The case $m(r)\equiv 1$ is trivial, since we can take $\beta=0$ in \eqref{eq:m(r)-formula}. So we assume $m(r)\not\equiv 1$ below. 

The proof is essentially the same as that of \cite[Theorem~2.3]{HKV}, with additional steps showing the irrationality of $\frac{\ln b}{\ln a}$ and the uniqueness of $\ab, a, b$. As in \cite{HKV}, we apply Lemma \ref{lem:circle-unwrap} with 
$$\varphi(s)=m(e^s),\alls$$
and $e^{i\beta_0}=m(1)$ to find a continuous function $\phi: \Rplus\rightarrow\R$, such that 
$\phi(1)=\beta_0$ and 
\begin{equation}\label{eq:phi-1}
m(r)=e^{i\phi(r)},\quad r>0.    
\end{equation}
In terms of $\phi$, \eqref{eq:m(r)} now reads 
\begin{equation}\label{eq:phi(r)}
    \begin{cases}
        \phi(ar)=2\phi(r)+2\pi M(r), \\
        \phi(br)=3\phi(r)+2\pi N(r),\quad r>0,\\
    \end{cases}
\end{equation}
where $M(r), N(r)\in \Z$. 
Since $\phi$ is continuous, it follows that $$M(r)\equiv M,\;\; N(r)\equiv N,\quad r>0,$$
for some constants $M, N\in \Z$. 
Apply \eqref{eq:phi(r)} to $\phi(abr)$ and $\phi(bar)$, respectively. We find that 
$$6\phi(r)+2(2\pi N)+(2\pi M)=
6\phi(r)+3(2\pi M)+(2\pi N).$$
Thus, 
\begin{equation}\label{eq:N=2M}
N=2M.    
\end{equation}
Following \cite{HKV}, we use \eqref{eq:N=2M} to rewrite \eqref{eq:phi(r)} as 
\begin{equation*}
    \begin{cases}
        \phi(ar)+2\pi M=2\big(\phi(r)+2\pi M\big), \\
        \phi(br)+2\pi M=3\big(\phi(r)+2\pi M\big),\quad r>0.\\
    \end{cases}
\end{equation*}
This shows that the function 
\begin{equation}\label{eq:phi-2}
\phi_1(r):=\phi(r)+2\pi M
\end{equation}
satisfies 
\begin{equation*}
    \begin{cases}
        \phi_1(ar)=2\phi_1(r), \\
        \phi_1(br)=3\phi_1(r), \quad r>0.\\
    \end{cases}
\end{equation*}
By iteration, it follows that we have 
\begin{equation}\label{eq:phi(abr)}
\phi_1(a^k b^\ell r)=2^k 3^\ell\,\phi_1(r),\quad r>0, 
\end{equation}
for all $k,\ell\in\Z$. 

Using \eqref{eq:phi(abr)}, one can show the following. Note that we have $a\neq 1$ (otherwise 
$m(ar)=m(r)^2$ would imply $m(r)\equiv1$). 

\begin{claim}\label{claim:irrational}
If $m(r)\not\equiv 1$, then $\frac{\ln b}{\ln a}\not\in\mathbb Q.$ 
\end{claim}

To prove the claim, assume for a contradiction that $\frac{\ln b}{\ln a}\in\mathbb Q$. Then we have $a^{k_0} b^{\ell_0} =1$ for some $(k_0,\ell_0)\in\Z^2\backslash\{(0,0)\}$. By \eqref{eq:phi(abr)}, this implies 
$$\phi_1(r)=2^{k_0} 3^{\ell_0}\,\phi_1(r),\quad r>0.$$
Since $2^{k_0} 3^{\ell_0}\neq 1$, it follows that $\phi_1(r)\equiv 0,$
and, consequently,  
$$m(r)=e^{i\phi(r)}=e^{i\phi_1(r)}\equiv 1,$$
This is a contradiction since $m(r)\not\equiv 1$. 

By Claim \ref{claim:irrational}, the set $\{a^k b^\ell: k,\ell\in\Z\}$ is dense in $\Rplus$. 
Thus, by \eqref{eq:phi(abr)} and the continuity of $\phi_1$, 
$\phi_1$ has no zero in $\Rplus$ (otherwise we would have $\phi_1\equiv 0$, which contradicts $m(r)\not\equiv 1$). 
So $\phi_1$ has a definite sign in $\Rplus$. 
Now consider the function 
\begin{equation*}
\widetilde\phi(s):=\ln|\phi_1(e^s)|,\alls.
\end{equation*}
In terms of $\widetilde\phi$, \eqref{eq:phi(abr)} reads 
\begin{equation}\label{eq:phi(s)-tilde}
\widetilde\phi(s+k\ln a+\ell \ln b)=\widetilde\phi(s)+k\ln 2+\ell \ln 3,\alls,\ k,\ell\in\Z.
\end{equation}
Denote 
$$Q=\{k\ln a+\ell \ln b: k,\ell\in\Z\}.$$
Then $Q$ is dense in $\R$. Moreover, by \eqref{eq:phi(s)-tilde}, we have for any $q\in Q$, 
$$\widetilde\phi(s+q)-\widetilde\phi(s)=\mathrm{const},\alls.$$
By Lemma \ref{lem:dense-steps}, there exist constants $\alpha,\gamma\in\R$, such that 
\begin{equation}\label{eq:phi(s)-tilde-2}
\widetilde\phi(s)=\alpha s+\gamma,\alls. 
\end{equation}
It follows that 
$$|\phi_1(r)|=e^\gamma e^{\alpha \ln r}=e^\gamma r^{\alpha},\quad r>0,$$
and thus,  
\begin{equation}\label{eq:phi-3}
\phi_1(r)=\beta r^{\alpha},\quad r>0,
\end{equation}
for some constant $\beta\in\Rstar$. 
Combining \eqref{eq:phi-1}, \eqref{eq:phi-2}, and \eqref{eq:phi-3}, we obtain 
$$m(r)=e^{i\beta r^{\alpha}},\quad r>0.$$ 
This proves \eqref{eq:m(r)-formula}. 

To show the uniqueness of $\ab$, suppose 
\begin{equation}\label{eq:m=m1}
m(r)=e^{i\beta r^{\alpha}}=e^{i\beta_1 r^{\alpha_1}},\quad r>0, 
\end{equation}
for some constants $\alpha, \beta, \alpha_1, \beta_1\in\R$. 
Since $m(r)\not\equiv 1$, we have $\beta, \beta_1\neq 0$ and $\alpha, \alpha_1\neq 0$ (otherwise $m(r)\equiv c$ would imply $c=1$ by \eqref{eq:m(r)}). Therefore, by \eqref{eq:m=m1} and Lemma \ref{lem:single-double}(ii), we have $\alpha=\alpha_1$ and $\beta=\beta_1$. 

To show the uniqueness of $a, b$, we use \eqref{eq:m(r)-formula} to rewrite \eqref{eq:m(r)} as 
\begin{align*}
    \begin{cases}
    e^{i\beta a^\alpha r^{\alpha}}=e^{i\beta 2 r^{\alpha}},\\
    e^{i\beta b^\alpha r^{\alpha}}=e^{i\beta 3 r^{\alpha}},\quad r>0.
    \end{cases}
\end{align*}
Since $\alpha, \beta\neq 0$, by Lemma \ref{lem:single-double}(ii) it follows that 
$$a^\alpha=2,\;\; b^\alpha=3,$$
and thus $a=2^{1/\alpha}$, $b=3^{1/\alpha}$. 
This completes the proof of Lemma \ref{lem:HKV}. 
\end{proof}

We can now prove Theorem \ref{thm:1}. 

\begin{proof}[Proof of Theorem \ref{thm:1}]
Since $T$ is translation-invariant, by Lemma \ref{lem:translation-inv} there exists a measurable function $m:\R\rightarrow\Sone$, such that 
\begin{equation}\label{eq:hat-Tf}
\widehat{Tf}(\xi)=m(\xi)\hat f(\xi),\allf.
\end{equation}
Since $T$ is symmetric and LUC under dilation, by Lemma \ref{lem:LUC-m} there exists a continuous function 
$\widetilde m:\Rplus\rightarrow\Sone$, such that 
\begin{equation}\label{eq:m(xi)}
m(\xi)=\widetilde m(|\xi|),\aexi.
\end{equation}
Since $T$ semistable, by Lemma \ref{lem:semistable-multiplier}(ii) and the continuity of $\widetilde m$, there exist constants $a,b>0$, such that 
$$\begin{cases}
\widetilde m(ar)=\widetilde m(r)^2,\\
\widetilde m(br)=\widetilde m(r)^3,\quad r>0. 
\end{cases}$$
By Lemma \ref{lem:HKV}, it follows that there exist constants $\ab\in\R$, such that 
\begin{equation}\label{eq:tilde-m}
\widetilde m(r)=e^{i\beta r^\alpha},\quad r>0.
\end{equation}
Combining \eqref{eq:hat-Tf}, \eqref{eq:m(xi)}, and \eqref{eq:tilde-m}, we see that 
\begin{equation}\label{eq:Tfxi-1}
\widehat{Tf}(\xi)=e^{i\beta|\xi|^\alpha}\hat f(\xi),\quad f\in\HH.
\end{equation}
This shows the existence of $\ab$. Note that if $T\neq I$, then $e^{i\beta r^\alpha}\not\equiv 1$, thus $\alpha,\beta\neq 0$. 

To show the uniqueness of $\ab$, suppose that $T\neq I$ and $\alpha_1, \beta_1\in\R$ are constants (possibly different from $\ab$) such that 
\begin{equation}\label{eq:Tfxi-2}
\widehat{Tf}(\xi)=e^{i\beta_1|\xi|^{\alpha_1}}\hat f(\xi),\quad f\in\HH.
\end{equation}
Comparing \eqref{eq:Tfxi-1} and \eqref{eq:Tfxi-2}, we see that 
\begin{equation*}
e^{i\beta r^{\alpha}}=e^{i\beta_1 r^{\alpha_1}},\quad r>0.
\end{equation*}
Since $\alpha, \alpha_1, \beta, \beta_1\neq 0$, 
by Lemma \ref{lem:single-double}(ii) it follows that 
$\alpha=\alpha_1$ and $\beta=\beta_1$. This shows the uniqueness of $\ab$. The proof of Theorem \ref{thm:1} is complete. 
\end{proof}




\section{Proof of Theorem \ref{thm:2}}
\label{sec:thm2}

In this section, we prove Theorem \ref{thm:2}. The proof relies on the following extension of Lemma \ref{lem:single-double}. 

\begin{lemma}
\label{lem:triple}
Suppose $\alpha_j\in\R$, $\beta_j\in\Rstar$, $j=1, 2, 3$. 
If 
\begin{equation}\label{eq:triple}
e^{i\beta_1 r^{\alpha_1}}e^{i\beta_2 r^{\alpha_2}}e^{i\beta_3 r^{\alpha_3}}\equiv 1,\ r>0,     
\end{equation}
then one of the following holds:\\
$(a)$\ $\alpha_1=\alpha_2=\alpha_3=0,\;\; \beta_1+\beta_2+\beta_3\in 2\pi\Z$;\\
$(b)$\ $\alpha_1=\alpha_2=\alpha_3\neq 0,\;\; \beta_1+\beta_2+\beta_3=0$;\\
$(c)$\ there exists $j_1\in\{1,2,3\}$, such that
$$\begin{cases}
\alpha_{j_1}=0,\;\;\beta_{j_1}\in 2\pi\Z,\;\;\text{and}\\
\alpha_{j_2}=\alpha_{j_3}\neq 0,\;\; \beta_{j_2}+\beta_{j_3}=0, 
\end{cases}$$
where $\{j_2,j_3\}=\{1,2,3\}\backslash\{j_1\}$. 
\end{lemma}

\begin{proof}
The proof reduces to examining the following three cases. 

{Case 1}: $\#\{\alpha_1,\alpha_2,\alpha_3\}=1$.  In this case, we are in cases (a)/(b) above, and the proof is the same as that of Lemma \ref{lem:single-double}. 

{Case 2}: $\#\{\alpha_1,\alpha_2,\alpha_3\}=2$. In this case, without loss of generality, we may assume that $\alpha_1\neq\alpha_2=\alpha_3$. Then \eqref{eq:triple} reads 
\begin{equation}\label{eq:double}
e^{i\beta_1 r^{\alpha_1}}e^{i(\beta_2+\beta_3) r^{\alpha_2}}\equiv 1,\quad r>0.  
\end{equation}
Since $\alpha_1\neq\alpha_2$ and $\beta_1\neq0$, by Lemma \ref{lem:single-double}(ii) we must have $\beta_2+\beta_3=0$. Thus, \eqref{eq:double} becomes
$$e^{i\beta_1 r^{\alpha_1}}\equiv 1,\quad r>0.$$
Since $\beta_1\neq0$, by Lemma \ref{lem:single-double}(i) we must have $\alpha_1=0$ and $\beta_{1}\in 2\pi\Z$. So we are in case (c) above. 

{Case 3}: $\#\{\alpha_1,\alpha_2,\alpha_3\}=3$. In this case, the set $$\{r>0: {\beta_1 r^{\alpha_1}}+{\beta_2 r^{\alpha_2}}+{\beta_3 r^{\alpha_3}}\in2\pi\Z\}$$ 
is countable. Thus \eqref{eq:triple} does not hold. 
\end{proof}

We are now ready to prove Theorem \ref{thm:2}. 

\begin{proof}[Proof of Theorem \ref{thm:2}]
If $T(t)\equiv I$, then \eqref{eq:T(t)-hat} trivially holds with $\beta=0$ (and any $\alpha\in\R$). 
So we assume that $T(t_0)\neq I$ for some $t_0\in\Rstar$ below. 
By applying the rescaling $t\mapsto t_0 t$, we may further assume that $t_0=1$. 

Apply Theorem \ref{thm:1} with $T=T(1)$. We see that, since $T(1)\neq I$, there exist unique constants $\ab\in\Rstar$, such that 
\begin{equation}\label{eq:T(1)-hat}
\widehat{T(1)f}(\xi)=e^{i\beta|\xi|^\alpha}\hat f(\xi),\quad f\in\HH. 
\end{equation}
To finish the proof, it suffices to show that \eqref{eq:T(t)-hat} holds with the constants $\ab$ above. 

Now apply Theorem \ref{thm:1} with $T=T(t)$, $t\in\R$. We have 
\begin{equation}\label{eq:T(t)-hat-3}
\widehat{T(t)f}(\xi)=e^{i\beta(t)|\xi|^{\alpha(t)}}\hat f(\xi),\quad f\in\HH, 
\end{equation}
for some $\alpha(t), \beta(t)\in\R$. 
On the other hand, by the group property \eqref{eq:group-property-0}, we have 
\begin{equation}\label{eq:group-property}
T(1)=T(t)T(1-t).
\end{equation}
Combining \eqref{eq:T(1)-hat}--\eqref{eq:group-property}, and by continuity, we have 
$$e^{i\beta r^{\alpha}}=e^{i\beta(t) r^{\alpha(t)}}e^{i\beta(1-t) r^{\alpha(1-t)}},\quad r>0.$$
Since $\ab\neq 0$, by Lemma \ref{lem:triple} one of the following holds: 
\begin{align*}
&(i)\ \ T(t)=I;\\
&(ii)\ T(t)\neq I,\;\;\alpha(t)=\alpha.
\end{align*}
In view of this dichotomy, define 
\begin{equation}\label{eq:beta-tilde}
\widetilde\beta(t)=
\begin{cases}
    0,\quad\quad\; \text{if }T(t)=I,\\
    \beta(t),\quad \text{if }T(t)\neq I.
\end{cases}
\end{equation} 
Then, by checking the two cases above, we have 
\begin{equation}\label{eq:beta-tilde-2}
e^{i\beta(t) r^{\alpha(t)}}=e^{i\widetilde\beta(t) r^{\alpha}},\quad r>0,\ t\in\R.
\end{equation} 
Using \eqref{eq:beta-tilde-2}, we can rewrite \eqref{eq:T(t)-hat-3} as 
\begin{equation}\label{eq:T(t)-hat-3.5}
\widehat{T(t)f}(\xi)=e^{i\widetilde\beta(t)|\xi|^{\alpha}}\hat f(\xi),\quad f\in\HH.  
\end{equation}
With \eqref{eq:T(t)-hat-3.5}, the group property \eqref{eq:group-property-0} implies
\begin{equation*}
e^{i\widetilde\beta(t_1+t_2) r^{\alpha}}=e^{i\widetilde\beta(t_1) r^{\alpha}}e^{i\widetilde\beta(t_2) r^{\alpha}},\quad r>0,\ t_1, t_2\in\R.
\end{equation*} 
Since $\alpha\neq 0$, by Lemma \ref{lem:single-double}(i) it follows that 
$\widetilde\beta(t)$ satisfies 
\begin{equation}\label{eq:beta-tilde-4}
\widetilde\beta(t_1+t_2)=\widetilde\beta(t_1)+\widetilde\beta(t_2),\quad t_1, t_2\in\R. 
\end{equation}

\begin{claim}
\label{claim:measurable}
The function $\widetilde\beta: \R\rightarrow\R$ is measurable. 
\end{claim}

Assume that the claim is true. Then, by Lemma \ref{lem:steinhaus}, \eqref{eq:beta-tilde-4} implies $\widetilde\beta(t)=\widetilde\beta(1)t.$
On the other hand, by \eqref{eq:T(1)-hat} and \eqref{eq:beta-tilde}, we have $\widetilde\beta(1)=\beta$. Thus, 
\begin{equation}\label{eq:beta-tilde-5}
\widetilde\beta(t)=\beta t,\quad t\in\R.
\end{equation}
Combining \eqref{eq:T(t)-hat-3.5} and \eqref{eq:beta-tilde-5}, we obtain \eqref{eq:T(t)-hat}. 

It remains to prove the claim, which is a consequence of the weak measurability of $\{T(t)\}_{t\in\R}$. 
Indeed, by Definition \ref{def:op-group}, \eqref{eq:T(t)-hat-3.5}, and Plancherel's theorem, for any $f, g\in \HH$, the function 
$$t\mapsto \langle T(t)f,g\rangle=\intR e^{i\widetilde\beta(t)|\xi|^{\alpha}}\hat f(\xi)\overline{\hat g(\xi)}d\xi$$
is measurable on $\R$. In particular, taking 
$$\hat f(\xi)=\hat g(\xi)=\begin{cases}
    \sqrt{N},\quad\text{if } r\le \xi< r+\frac1N,\\
    0,\quad\quad\;\text{otherwise},
\end{cases}$$
and letting $N\rightarrow\infty$, we see that for any fixed $r>0$, the function 
$$t\mapsto e^{i\widetilde\beta(t)r^{\alpha}}$$
is measurable on $\R$. 
Since $\alpha\neq0$, $r^\alpha$ can be chosen arbitrarily small. 
Combining this with the identity 
$$\lim_{\varepsilon\rightarrow0}\frac{e^{i\widetilde\beta(t)\varepsilon}-1}{i\varepsilon}=\widetilde\beta(t),\quad t\in\R,$$
we see that $\widetilde\beta(t)$ is measurable on $\R$. 
This proves the claim, 
and the proof of Theorem \ref{thm:2} is complete. 

\end{proof}







\bibliographystyle{abbrv}
\bibliography{bibliography}

@article{Hormander,
 author = {H{\"o}rmander, Lars},
 title = {Estimates for translation invariant operators in {{\(L^p\)}} spaces},
 fjournal = {Acta Mathematica},
 journal = {Acta Math.},
 issn = {0001-5962},
 volume = {104},
 pages = {93--140},
 year = {1960},
 language = {English},
 doi = {10.1007/BF02547187},
 zbMATH = {3151901},
 Zbl = {0093.11402}
}

@article{HKV,
 author = {Hamedani, G. G. and Key, Eric S. and Volkmer, Hans},
 title = {Solution to a functional equation and its application to stable and stable-type distributions},
 fjournal = {Statistics \& Probability Letters},
 journal = {Stat. Probab. Lett.},
 issn = {0167-7152},
 volume = {69},
 number = {1},
 pages = {1--9},
 year = {2004},
 language = {English},
 doi = {10.1016/j.spl.2004.02.010},
 zbMATH = {2159995},
 Zbl = {1061.62017}
}

\end{document}